\documentclass[a4paper]{article}

\usepackage{a4wide}

\usepackage{graphicx}
\usepackage[T1]{fontenc}

\usepackage{amssymb}
\usepackage{amsmath}
\usepackage{amsthm}
\usepackage{bm}
\usepackage[linesnumbered,ruled,vlined]{algorithm2e}
\usepackage{tikz}
\usetikzlibrary{decorations.pathreplacing}
\usepackage{enumerate}
\usepackage{pgfplots}
\pgfplotsset{compat=1.14}

\newtheoremstyle{it_dotless} 
                        {0.5em}   
                        {0.5em}   
                        {\itshape}  
                        {}          
                        {\bfseries} 
                        {:}         
                        {\newline}  
                        {}			

\newtheoremstyle{dotless} 
                        {0.5em}   
                        {0.5em}   
                        {}  		
                        {}          
                        {\bfseries} 
                        {:}         
                        {\newline}  
                        {}			
                        
\theoremstyle{it_dotless}
\newtheorem{theorem}{Theorem}
\newtheorem{lemma}[theorem]{Lemma}

\theoremstyle{dotless}
\newtheorem{remark}[theorem]{Remark}

\newtheorem{definition}[theorem]{Definition}

\newcommand{\bitem}{\begin{itemize}}
\newcommand{\eitem}{\end{itemize}}
\newcommand{\mc}[1]{\mathcal{#1}}

\newcommand{\N}{\mathbb{N}}
\newcommand{\R}{\mathbb{R}}

\newcommand{\bpm}{\begin{pmatrix}}
\newcommand{\epm}{\end{pmatrix}}
\newcommand{\bsm}{\left(\begin{smallmatrix}}
\newcommand{\esm}{\end{smallmatrix}\right)}
\newcommand{\T}{\top}

\newcommand{\la}{\langle}
\newcommand{\ra}{\rangle}

\DeclareMathOperator{\tr}{tr}
\DeclareMathOperator{\Diag}{Diag}
\DeclareMathOperator{\diag}{diag}

\DeclareMathOperator{\intr}{int}

\DeclareMathOperator{\bd}{bd}

\DeclareMathOperator{\argmin}{arg min}
\DeclareMathOperator{\argmax}{arg max}

\DeclareMathOperator{\aff}{aff}
\DeclareMathOperator{\conv}{conv}

\DeclareMathOperator*{\keq}{\leq_\mathcal{K}}
\DeclareMathOperator*{\eqk}{\geq_\mathcal{K}}

\DeclareMathOperator*{\diam}{diam}
\DeclareMathOperator*{\dist}{dist}

\DeclareMathOperator*{\K}{\mc{K}}

\newcommand{\vect}[1]{\bm{#1}}
\newcommand{\matr}[1]{\mathbf{#1}}
\newcommand{\I}{\;\middle | \;}
\newcommand{\e}[1][]{\vect{e}_{#1}}		

\usepackage{hyperref} 	
\hypersetup{
    colorlinks=true,
    linkcolor=blue,
    filecolor=magenta,      
    urlcolor=cyan,
    citecolor=red
}

\begin{document}

\title{
The LP-Newton Method and Conic Optimization 
\thanks{
Authors gratefully acknowledge support by the German Research Foundation (DFG), grant GRK 1653, and the German-Japanese University Network (HeKKSaGOn). We thank Prof. Fujishige for introducing us to the topic during a visit to Kyoto. His comments also helped to greatly improve the present paper.
}
}

\author{Francesco Silvestri
 \footnote{Institut f\"{u}r Informatik, Heidelberg University, INF205, 69120 Heidelberg, Germany}$^{\; ,}$\footnote{IWR, Heidelberg University, INF205, 69120 Heidelberg, Germany} \and Gerhard~Reinelt \footnotemark[2]}


%

\maketitle

\begin{abstract}
We propose that the LP-Newton method can be used to solve conic LPs over a \emph{conic box}, whenever linear optimization over an otherwise unconstrained conic box is easy. In particular, if $\keq$ is the partial order induced by a proper convex cone $\K$, then optimizing a linear function over the intersection of $[\vect{l},\vect{u}]_{\K}=\{ \vect{l}\keq \vect{x}\keq \vect{u}\}$ and an affine subspace can be done with this method whenever optimizing a linear function over $[\vect{l},\vect{u}]_{\K}$ is efficient. 

This generalizes the result for the case of $\K=\R^n_+$ that was originally proposed for using the method. Specifically, we show how to adapt this method for both SOCP and SDP problems and illustrate the method with a few experiments. While the approach is promising due to the low amount of Newton steps needed, solving the minimum-norm-point problem involved
in the Newton step with a Frank-Wolfe algorithm is not advisable.
\end{abstract}

\section{Introduction}

\subsection{Summary}

The LP-Newton method was introduced in \cite{Fujishige08} to find an end-point of the intersection of a line and a zonotope. Here L and P, respectively, stand for line and (convex) polyhedron, not for Linear Programming.

The algorithm resembles Dinkelbach's discrete Newton(-Raphson) method for finding the zero point of a one-dimensional piecewise-linear convex function, where the minimum-norm-point algorithm is utilized to compute a subderivative of this function. 

The minimum-norm-point algorithm is fast in this setting since its subroutine consists of linear optimization over a zonotope, which is trivial and can be done in linear time. This naturally leads to an algorithm for the zonotope formulation of linear programming problems.

In this paper, we extend the notion of zonotopes to $\K$-zonotopes for any proper convex cone $\K$, which we will define as the image of a conic interval under a linear transformation. This leads to a $\K$-zonotope formulation for conic LPs which can then be solved by the LP-Newton method in the same manner. While this makes the method much more general, we lose the finite convergence and are left with asymptotic convergence instead. 

%
%
%

\subsection{Background}

This paper is based on the original work \cite{Fujishige08} where the LP-Newton method was introduced and successfully applied to linear programming. 
Consequently, our extension also heavily relies on the minimum-norm-point algorithm, which is a special case of the more general class of Frank-Wolfe algorithms \cite{Jaggi13,Wolfe76} and will be discussed in more detail in section~\ref{sec:minimum-norm-point}.

The Newton method itself has long been an effective tool in both continuous optimization as well as cone programming through the use of interior point methods  \cite{Polyak07}. In fact, the field of cone programming is currently dominated by interior point methods, which have been extensively studied in the recent years \cite{Todd08}. However, the complexity of these algorithms is usually prohibitive for large problems and so there is a huge demand for an alternative with better complexity properties. 

To this end, one line of research focuses on first-order methods for this kind of problem \cite{Esser10}. In particular, a combination of operator splitting and homogeneous self-dual embedding was recently proposed \cite{Boyd16} and was shown to beat state-of-the-art interior point methods on large instances.

Since our proposed method also falls into the class of first-order methods, the main goal of this paper is to assess whether it is able to compete with interior point methods as well.

\subsection{Organization of the paper}
The paper consists of two parts.

In the first part, we restate the results from \cite{Fujishige08} in the setting of conic LPs. For this, we introduce conic zonotopes in section 2 and show how to adapt the LP-Newton method in section~3. In section 4, we recall how to use the minimum-norm-point algorithm for the projection step of the LP-Newton method.

In the second part, we look at widespread convex cones and how they interact with the proposed method. In section 5, we give some conditions for $\K$ to be exploited by the algorithm and consider the nonnegative orthant $\R^n_+$, the Lorentz-cone $\mc{L}_n$ and the cone of positive semidefinite matrices $\mc{S}^n_+$. Finally, section 6 reports about some experiments on $\mc{L}_n$ and $\mc{S}^n_+$, illustrating the behaviour of the algorithm in a setting which was not considered in the original paper \cite{Fujishige08}.

\section{Preliminaries}

\subsection{Conic zonotopes}

Throughout this paper, let $\K\subseteq \R^n$ be a proper convex self-dual cone. We can then define a partial order $\keq$ on $\R^n$ by demanding for all $\vect{x},\vect{y}\in \R^n$ that 
\begin{equation*}
 \vect{y}\keq \vect{x} \quad \Leftrightarrow\quad \vect{x}-\vect{y}\in \K,
\end{equation*}
where $\vect{0}\keq \vect{x}$ is short for membership $\vect{x}\in\K$ and extend 
\begin{equation*}
 \vect{y}<_{\K} \vect{x} \quad \Leftrightarrow\quad \vect{x}-\vect{y}\in \K,\quad \vect{y}\neq \vect{x}.
\end{equation*}

\noindent For any two points $\vect{l}\keq \vect{u}\in \R^n$, denote their interval with respect to $\keq$ by 
\begin{equation*}
[\vect{l},\vect{u}]_{\K}:=\left\{ \vect{x}\in \R^n \,\middle|\, \vect{l} \keq \vect{x} \keq \vect{u} \right\}.
\end{equation*}
Throughout this paper, we will assume that $l<_{\K} u$, which implies that
\begin{equation*}
\conv(\{\vect{l},\vect{u}\})\subseteq [\vect{l},\vect{u}]_{\K}
\end{equation*}
is nonempty and nontrivial. Now let $\matr{A}\in \R^{m\times n}$ and define 
\begin{equation*}
Z = \left\{ \vect{z} \,\middle|\, \vect{z}=\matr{A}\vect{x}, \vect{x}\in [\vect{l},\vect{u}]_{\K}\right\},
\end{equation*}
which we will call a \emph{$\K$-zonotope} in $\R^m$. This generalizes the already established \emph{zonotopes} \cite{Ziegler95}, which can be defined as $\R_+^n$-zonotopes in our setting.

\begin{remark}

Optimizing a linear function $\vect{c}$ over $Z$ is just as hard as optimizing a linear function over $[\vect{l},\vect{u}]_{\K}$, since
\begin{equation*}
\max\left\{ \la \vect{c},\vect{z}\ra \,\middle|\,\vect{z}\in Z\right\}   = \max\left\{ \la \matr{A}^\T \vect{c}, \vect{x}\ra \I \vect{x}\in [\vect{l},\vect{u}]_{\K} \right\}
\end{equation*}

\end{remark}

\subsection{CLP reformulation}

The standard form of a conic linear program, or CLP for short, is
\begin{equation*}
\max \left\{ \la \vect{c},\vect{x}\ra \I  \matr{A}\vect{x}=\vect{b},\; \vect{x}\in \K \right\}.
\end{equation*}
In the following, we will instead consider the following CLP 
\begin{equation}
\max \left\{ \la \vect{c},\vect{x}\ra \I  \matr{A}\vect{x}=\vect{b},\; \vect{x}\in [\vect{l},\vect{u}]_{\K} \right\},\tag{Box-CLP} \label{eq:CLP}
\end{equation}
which is slightly more restrictive than the standard form. Given appropriate bounds on the feasible region, which are often available or can be easily computed, the standard form can be reformulated into the form of \eqref{eq:CLP}.

\noindent Following \cite{Fujishige08}, we first encode the constraint by defining an $(m+1)\times n$ matrix 
\begin{equation*}
\bar{\matr{A}}=\bpm \matr{A} \\ \vect{c}^\T \epm
\end{equation*}
together with a $\K$-zonotope
\begin{equation*}
\bar{Z} = \left\{ \vect{z} \I \vect{z}=\bar{\matr{A}}\vect{x},\;  \vect{x}\in [\vect{l},\vect{u}]_{\K}\right\}
\end{equation*}
and the line 
\begin{equation*}
L = \left\{ \bpm \vect{b}\\ \gamma\epm  \,\middle|\, \gamma\in \R \right\}.
\end{equation*}
Using this notation, \eqref{eq:CLP} can be restated as
\begin{equation*}
\gamma^*=\max\left\{ \gamma  \I  \bpm \vect{z} \\ \gamma \epm  \in L\cap \bar{Z}\right\}, \tag{CLP$'$}\label{eq:CLP2}
\end{equation*}
where $\gamma\in \R$. The rest of this paper will be concerned with the question of how to solve \eqref{eq:CLP2}.


\section{The Method}

Let 
\begin{equation}
\gamma_0 = \max\left\{ \la \vect{c},\vect{x}\ra \I \vect{x}\in [\vect{l},\vect{u}]_{\K}\right\}, \label{eq:gamma_0}
\end{equation}
which is clearly an upperbound for \eqref{eq:CLP2}. Furthermore, for any closed convex set $C\subseteq\R^n$, let $\pi_C:~\R^n\rightarrow C$ denote the Euclidean projection onto $C$ given by
\begin{equation*}
\pi_C(\vect{x})=\argmin \left\{  \|\vect{x}-\vect{y}\|_2 \I \vect{y}\in C \right\}
\end{equation*}
and define
\begin{equation*}
\bar{\vect{b}}(\gamma)=\bpm \vect{b}\\ \gamma\epm\quad \forall \gamma\in \R
\end{equation*}
to parametrize $L$ and simplify notation. 
Next, consider the continous, convex scalar function $g: \R \rightarrow \R$ given by 
\begin{equation*}
\gamma \mapsto g(\gamma) := \left\| \bar{\vect{b}}(\gamma) - \pi_{\bar{Z}}(\bar{\vect{b}}(\gamma))\right\|_2 = \dist(\bar{\vect{b}}(\gamma),\bar{Z}).
\end{equation*}
Then by definition, $g(\gamma)=0$ if and only if $\bar{\vect{b}}(\gamma)\in \bar{Z}$, and the set of zeros of $g$ parametrizes the feasible set $L\cap \bar{Z}$ of \eqref{eq:CLP2} via $\gamma \mapsto \bar{\vect{b}}(\gamma)$. In particular, the zeros of $g$ coincide with the values that are attained by the objective of \eqref{eq:CLP2}, and we get the following characterization. 

\begin{lemma}
The optimal value $\gamma^*$ of \eqref{eq:CLP2} is given as the maximal zero
\begin{equation}
\max\left\{ \gamma\in \R \I g(\gamma)=0\right\}. \tag{MZ}\label{eq:MZ}
\end{equation}
\end{lemma}

The idea of the CLP-Newton method is therefore to use $\gamma_0$ as the starting point for a generalized Newton method that solves \eqref{eq:MZ}. Since $g$ is convex, this ensures that we will converge towards $\gamma^*$ from above. In particular, since $g$ is in general non-differentiable, a subdifferential must be used instead of the usual derivative, which can be extracted from a projection onto $\bar{Z}$, as shown in the following lemma.

\begin{lemma}\label{lem:subdifferential_g}
For all $\gamma\in \R$ we have  
\begin{equation*}
\frac{ \gamma -\left\la \e[m+1],\pi_{\bar{Z}}(\vect{\bar{b}}(\gamma))\right\ra }{\dist\big(\vect{\bar{b}}(\gamma),\bar{Z}\big)} \in \partial g(\gamma).
\end{equation*}
\end{lemma}
\begin{proof}
The function $g(\gamma)=\dist\big(\vect{\bar{b}}(\gamma),\bar{Z}\big)$ is the composition of an affine map
\begin{equation*}
\vect{\bar{b}}(\gamma) = \bpm \vect{0} \\ 1 \epm  \gamma + \bpm \vect{b} \\ 0 \epm
\end{equation*}
and a distance function, so we can use the chain rule for subderivatives \cite[Ch.~A]{Rockafellar09} to yield
\begin{equation*}
\partial g(\gamma) = \e[m+1]^\T \cdot \partial_{\vect{x}=\vect{\bar{b}}(\gamma)} \dist\big(\vect{x},\bar{Z}\big).
\end{equation*}
For any convex set $C$, we have
\begin{equation*}
\frac{\vect{x}-\pi_{C}(\vect{x})}{\dist\big(\vect{x},C\big)} \in \partial \dist(\vect{x}, C),
\end{equation*}
and using $C=\bar{Z}$ shows
\begin{equation*}
\frac{ \gamma -\left\la \e[m+1],\pi_{\bar{Z}}(\vect{\bar{b}}(\gamma))\right\ra }{\dist\big(\vect{\bar{b}}(\gamma),\bar{Z}\big)}
=
\frac{\left\la \e[m+1], \vect{\bar{b}}(\gamma)-\pi_{\bar{Z}}(\vect{\bar{b}}(\gamma))\right\ra }{\dist\big(\vect{\bar{b}}(\gamma),\bar{Z}\big)} \in \partial g(\gamma).
\end{equation*}
\end{proof}

With the extraction of a point from the subdifferential of $g$ taken care of, we can state the method.

\begin{algorithm}[H]\label{alg:CPN}
\caption{The CLP-Newton Method (CLPN)} 
\SetKw{NOT}{not}
\SetKw{OR}{or}

\SetKwFunction{F}{F}

\KwData{Data $\matr{A}, \vect{b}, \vect{c}, \vect{l}, \vect{u}$ for \eqref{eq:CLP2}, error tolerance $\varepsilon$.}
\KwResult{Approximate solution $\vect{x}_k$ or detection of infeasibility of \eqref{eq:CLP2}.}

Compute $\gamma_0 = \max\left\{ \la \vect{c},\vect{x}\ra \I \vect{x}\in [\vect{l},\vect{u}]_{\K}\right\}$\;

\For{$k=1,2,\ldots$}{
Find $\vect{x}_k$ such that $\bar{\matr{A}}\vect{x}_k=\pi_{\bar{Z}}(\bar{\vect{b}}(\gamma_{k-1}))$\;
\If{$\left\|\bar{\matr{A}}\vect{x}_k - \bar{\vect{b}}(\gamma_{k-1})\right\|_2 < \varepsilon$}{\Return{$\vect{x}_k$}\;}
Set $(\vect{z}_k^\T, \zeta_k)^\T = \bar{\matr{A}}\vect{x}_k$\;
\If{$\zeta_k \geq \gamma_{k-1}$}{\Return{``\eqref{eq:CLP2} is infeasible''}}
Compute $\gamma_{k}=\zeta_k- \| \vect{b}-\vect{z}_k \|^2_2 /(\gamma_{k-1}-\zeta_k)$\;
}
\end{algorithm}

Algorithm~\ref{alg:CPN} starts by initializing $\gamma_0$ and then proceeds to iterate by first checking stopping criteria and then performing a Newton step. In the $k$-th step, the current iterate $\gamma_{k-1}$ is processed as follows.  

After computing $g(\gamma_{k-1})$, the value $\gamma_{k-1}$ is accepted as a zero of $g$ if $g(\gamma_{k-1})\in [0,\varepsilon)$ for a given precision $\varepsilon$, thus terminating the algorithm. Otherwise, if the preliminary result $\zeta_k$ indicates that the minimum of $g$ was passed, then $g$ has no zeros, and the algorithm terminates with infeasibility.

Only if neither of these conditions is satsified, a new iterate $\gamma_k$ can be computed by performing a Newton step. This is done by choosing $h_k\in \partial g(\gamma_{k-1})$ as in Lemma \ref{lem:subdifferential_g} in the recursion 
\begin{equation*}
\gamma_k = \gamma_{k-1} - \frac{g(\gamma_{k-1})}{h_k}
\end{equation*}
to get 
\begin{equation}
\gamma_k = \gamma_{k-1} - \frac{\dist(\vect{\bar{b}}(\gamma_{k-1}), \bar{Z})^2}{\gamma_{k-1}-\zeta_k} = \zeta_k- \frac{\| \vect{b}-\vect{z}_k \|^2_2}{\gamma_{k-1}-\zeta_k}. \label{eq:gamma_update}
\end{equation}

The Newton step can also be understood geometrically by noting that
\begin{equation*}
H_k= \left\{ \bpm \vect{z} \\ \zeta \epm \in \R^{m+1} \I 
\left\la \bpm \vect{z} \\ \zeta \epm - \bpm \vect{z}_k \\ \zeta_k \epm , \bpm \vect{b} \\ \gamma_{k-1} \epm - \bpm \vect{z}_k \\ \zeta_k \epm \right\ra =0    \right\}
\end{equation*}
is a supporting hyperplane of $\bar{Z}$ at $\pi_{\bar{Z}}(\bar{b}(\gamma_{k-1}))$. 
As a consequence, any feasible point in $L\cap \bar{Z}$ is contained in the halfspace defined by $H_k$ which does not contain $\vect{\bar{b}}(\gamma_{k-1})$, and \eqref{eq:gamma_update} computes the intersection $L\cap H_k$. Figure \ref{fig:newton_step} compares how both approaches arrive at \eqref{eq:gamma_update}.

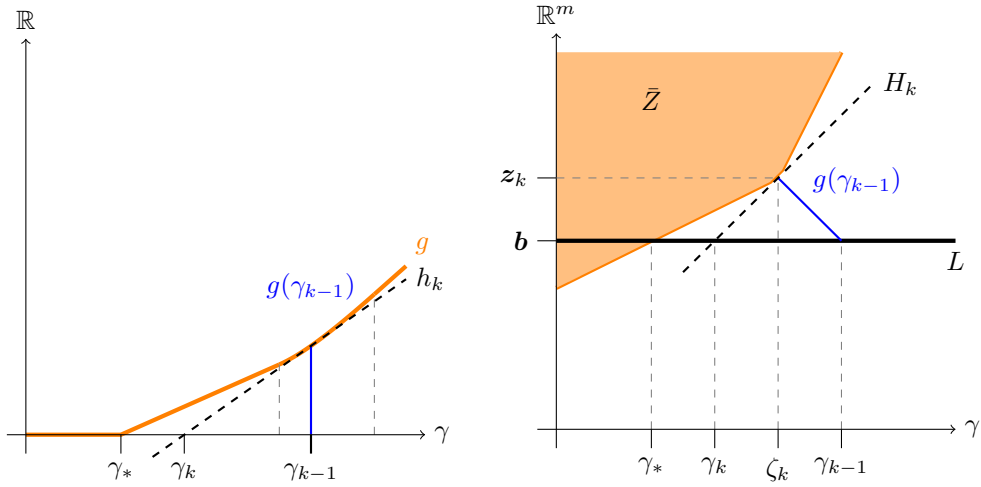
\begin{figure}[!htb]
\centering
\begin{tabular}{c c}

\begin{tikzpicture}[scale=2.5,domain=1:3]

    \draw[ultra thick, color=orange, domain=1:7/3] 
     plot (\x,{max( 0 , 1/sqrt(5)*(\x-1.5) , 2/sqrt(5)*(\x-2)  )}); 
    \draw[ultra thick, color=orange, domain=7/3:17/6] 
     plot (\x,{sqrt(173/36-13/3*\x+\x*\x)});
    \draw[ultra thick, color=orange, domain=17/6:3] 
     plot (\x,{2/sqrt(5)*(\x-2)});
    \draw (3,1) node[right, color=orange] {$g$};
	\draw [dashed, color=gray](7/3,0) -- (7/3,{sqrt(5)/6});
	\draw [dashed, color=gray](17/6,0) -- (17/6,{sqrt(5)/3});
	\draw [thick, color=blue](2.5,0) -- (2.5,{sqrt(2)/3});
    \draw [thick, dashed] (10/6,{-sqrt(2)/12}) -- (3,{sqrt(2)*7/12}) node [right] {$h_k$};
	\draw (2.5,2/3) node[above, color=blue] {$g(\gamma_{k-1})$}; 
	\draw [thick] (2.5,0) -- (2.5,-0.1) node[below] {$\gamma_{k-1}$};
	\draw (11/6,0) -- (11/6,-0.1) node[below] {$\gamma_{k}$};			
	\draw (1.5,0) -- (1.5,-0.1) node[below] {$\gamma_*$};
    \draw[->] (0.9,0) -- (3.1,0) node[right] {$\gamma$};
    \draw[->] (1,-0.1) -- (1,2.1) node[above] {$\R$};
\end{tikzpicture}

&

\begin{tikzpicture}[scale=2.5,domain=1:2.5]

    \draw[ultra thick, color=orange] plot (\x,{max(0.5*\x+0.25, 2*\x-3)}); 
	\fill [color=orange!50, domain=1:2.5, variable=\x]
      (1, 2)
      -- plot (\x,{max(0.5*\x+0.25, 2*\x-3)})
      -- (2.5, 2)
      -- cycle;    
    \draw (1.5,1.75) node {$\bar{Z}$};
	\draw [dashed, color=gray](1.5,0)  -- (1.5,1);
	\draw [dashed, color=gray](11/6,0) -- (11/6,1);	
    \draw [dashed, color=gray](13/6,0) -- (13/6,4/3);	
	\draw [dashed, color=gray](2.5,0)  -- (2.5,1);	
	\draw (1.5,0)  -- (1.5,-0.1)  node[below] {$\gamma_*$};
	\draw (11/6,0) -- (11/6,-0.1) node[below] {$\gamma_{k}$};	
	\draw (13/6,0) -- (13/6,-0.1) node[below] {$\zeta_{k}$};		
    \draw (2.5,0)  -- (2.5,-0.1)  node[below] {$\gamma_{k-1}$};	
    \draw[->] (1,-0.1) -- (1,2.1) node[above] {$\R^m$};    
    \draw [ultra thick]        (1,1)   -- (3.1,1)    node[below] {$L$};
    \draw                      (1,1)   -- (0.9,1)    node[left]  {$\vect{b}$};
    \draw [dashed, color=gray] (1,4/3) -- (13/6,4/3);
    \draw                      (1,4/3) -- (0.9,4/3)  node[left]  {$\vect{z}_k$};
    \draw[->] (0.9,0) -- (3.1,0) node[right] {$\gamma$};
    \draw (2.3,1.2) node[above right, color=blue] {$g(\gamma_{k-1})$};    
    \draw [thick, color=blue](2.5,1) -- (13/6, 4/3);
    \draw [dashed , thick] (10/6, 5/6) -- (16/6, 11/6) node[right] {$H_k$};    
\end{tikzpicture}


\end{tabular}

\caption[Update step for the Conic LP-Newton Method]
{Visualization of the update step \eqref{eq:gamma_update}. Left: Newton step using $g$. Right: Geometric deduction from supporting hyperplane $H_k$.}
\label{fig:newton_step}
\end{figure}

The following results from \cite{Fujishige08} are still valid. 

\begin{lemma}[\cite{Fujishige08}]
The following statements hold for all values of $k$ attained in Algorithm~\ref{alg:CPN}.
\begin{enumerate}[(i)]
\item If $\gamma_k> \zeta_k$, then $\zeta_k > \gamma_{k+1}$.
\item If $\gamma_k< \zeta_k$ or $\gamma_k= \zeta_k$ and $\vect{z}_k \neq \vect{b}$, then (CLPN) correctly assesses infeasibility of \eqref{eq:CLP2}.
\item If $\gamma_k = \zeta_k$ and $\vect{z}_k=\vect{b}$, then $\gamma_k$ is equal to the optimal value of \eqref{eq:CLP2}.
\end{enumerate}

\end{lemma}
%
%
%
\begin{remark}
Since $g$ is convex, Algorithm~\ref{alg:CPN} falls into the class of generalized Newton methods, which immediately shows asymptotic convergence in case that \eqref{eq:CLP2} is feasible and finite termination in case that there is no feasible solution. 
\end{remark}


\section{The Minimum-Norm-Point Algorithm}\label{sec:minimum-norm-point}

Algorithm~\ref{alg:CPN} can only be as efficient as its routines to compute $\gamma_0$ and $\pi_{\bar{Z}}$, and we would like to make sure that both operations can be done fairly efficient. To this end, we will present Algorithm~\ref{alg:FW}, a version of the minimum-norm-point algorithm adapted from  \cite{Bach11,Jaggi13}, in order to reduce the computation of $\pi_{\bar{Z}}$ to several CLPs like \eqref{eq:gamma_0}.

\begin{algorithm}[H]\label{alg:FW}
\caption{Minimum-Norm-Point Algorithm MNP for $\dist(y,\bar{Z})$ } 
\SetKw{NOT}{not}
\SetKw{OR}{or}

\SetKwFunction{F}{F}

\KwData{Data $\vect{l}, \vect{u}, \bar{\matr{A}}$ for \eqref{eq:CLP2}, $\bar{\vect{b}}\in \R^{m+1}$, $\vect{s}_0\in [\vect{l},\vect{u}]_{\K}$, error tolerance $\varepsilon$.}
\KwResult{$\hat{\vect{x}}\in [\vect{l},\vect{u}]_{\K}$ such that $0\leq \|\bar{\matr{A}}\hat{\vect{x}}-\bar{\vect{b}}\|_2 -\dist(\bar{\vect{b}},\bar{Z})\leq \varepsilon$.}

Set $P=\{\vect{s}_0\}$ and $k=0$\;

\For{$k=1,2,\ldots$}{
Compute $\vect{x}_k=\argmin\left\{ \|\bar{\matr{A}}\vect{x} -\bar{\vect{b}}\|_2^2 \,\middle|\, \vect{x}\in \aff(P)\right\}$\;
\If{$\vect{x}_k\in \conv(P)$}
{
Compute $\vect{s}_k = \argmin\left\{ \la \vect{s}, \bar{\matr{A}}^\T(\bar{\matr{A}}\vect{x}_{k}-\bar{\vect{b}})\ra \I  \vect{s}\in [\vect{l},\vect{u}]_{\K}\right\}$\;
\If{ $2 \la \bar{\matr{A}}(\vect{x}_{k}-\vect{s}_k), \bar{\matr{A}}\vect{x}_{k}-\bar{\vect{b}}\ra < \varepsilon$}{\Return{$\vect{x}_{k}$}\;}
\Else{Set $P = P\cup \{\vect{s}_k\}$\;}
}
\Else{
Compute $\hat{\lambda} = \max\left\{ \lambda \,\middle|\, \lambda\in[0,1],\;  \vect{x}_{k-1}+  \lambda (\vect{x}_k - \vect{x}_{k-1})\in \bar{Z}\right\}$\;
Set $\vect{x}_k = \vect{x}_{k-1}+  \hat{\lambda}(\vect{x}_k - \vect{x}_{k-1})$\;
Set $P$ to the minimal subset $P'\subseteq P$ such that $\vect{x}_k\in \conv(P')$\;
}
}
\end{algorithm}

While the exact number of iterations for MNP is an open problem, finite termination is established. We cite the following theorem from the survey  \cite{Jaggi13} about the more general class of Frank-Wolfe algorithms.

\begin{theorem}
Algorithm~\ref{alg:FW} produces a sequence $\{\vect{x}_k\}_{k\in \N}$ such that for $k>1$ and
\begin{equation*}
h_k := 2 \la \bar{\matr{A}}(\vect{x}_{k-1}-\vect{s}_k), \bar{\matr{A}}\vect{x}_{k-1}-\vect{y}\ra,
\end{equation*}
we get
\begin{equation*}
0 \leq \| \bar{\matr{A}}\vect{x}_k - \vect{y} \|_2 - \dist(\vect{y},\bar{Z}) \leq h_{k+1} \leq \frac{27\diam(\bar{Z})^2}{4(k+2)}.
\end{equation*}
In particular, the algorithm works correctly and terminates after a finite number of steps. 
\end{theorem}

\begin{remark}
The preceeding theorem also applies to other variants of the Frank-Wolfe algorithm in \cite{Jaggi13} that are able to approximate $\pi_{\bar{Z}}(\bar{\vect{b}}(\gamma_k))$. 
However, preliminary experiments have shown that the minimum-norm-point algorithm was the fastest algorithm for our purpose. For another discussion of MNP, consider \cite[Sct.~9.2]{Bach11}.
\end{remark}

We can now formalize when linear optimization over $[\vect{l},\vect{u}]_{\K}$ is ``sufficiently easy''.

\begin{definition}
$\K\subseteq \R^n$ is called \emph{suitable for the \textnormal{CLP}-Newton} method if 
the problem
\begin{equation}
\max\left\{ \la \vect{c},\vect{x}\ra \,\middle|\, \vect{x}\in [\vect{l},\vect{u}]_{\K}\subseteq \R^n  \right\} \label{eq:assumption}
\end{equation}
can be solved in time $\mathcal{O}(n^2)$.
\end{definition}

The idea behind this definition is that solving \eqref{eq:assumption} should be cheaper than solving a linear system of size $n\times n$. This way, \eqref{eq:assumption} not only takes care of $\gamma_0$, but also of $\pi_{\bar{Z}}$, since the bottleneck of the minimum-norm-point algorithm is the computation of $\vect{x}_k=\pi_{\aff(P)}(\bar{\vect{b}})$, corresponding to solving a linear system of size at most $n\times n$.

\begin{remark}
If $\{\K_i\}_{i\in I}$ is a family of cones suitable for the CLP-Newton method, so is their Cartesian product $\K=\bigotimes_{i\in I} \K_i$. In particular, such a $\K$-zonotope decomposes into several $\K_i$-zonotopes, so that \eqref{eq:assumption} can be solved in parallel for each $\K_i$, which makes the algorithm potentially much faster. 
\end{remark}


\section{Linear Optimization over K-Zonotopes}

In this section, we will give necessary conditions for \eqref{eq:assumption} being easy to solve depending on $\K$ and apply them to exemplary classes of cones. 

\subsection{Necessary conditions}

Looking at \eqref{eq:assumption}, we can make the following observation, where we will use $\bd(C)$ and $\intr(C)$ to respectively denote the boundary and interior of $C\subseteq \R^n$.

\begin{lemma}[Extreme points of {$[\vect{l},\vect{u}]_{\K}$}] \label{lem:boundary}
Problem \eqref{eq:assumption} is equivalent to
\begin{equation*}
\max\left\{ \la \vect{c},\vect{x}\ra \I \vect{x}\in \{\vect{l},\vect{u}\}\cup \left(\bd(\vect{l}+\K)\cap \bd(\vect{u}-\K)\right) \right\}.
\end{equation*}
\end{lemma}
\begin{proof}
Since $[\vect{l},\vect{u}]_{\K}$ is convex, any optimal solution $\vect{x}^*$ will necessarily belong to 
\begin{equation*}
\bd([\vect{l},\vect{u}]_{\K})= 
\big(\bd(\vect{l}+\K)\cap (\vect{u}-\K)\big)  \cup 
\big((\vect{l}+\K)\cap \bd(\vect{u}-\K)\big).
\end{equation*}
Now assume w.l.o.g. that $\vect{x}^* \in \bd(\vect{l}+\K)\cap \intr(\vect{u}-\K)\setminus \{\vect{l}\} $, so that we can write $\vect{x}^* = \vect{l} + \vect{y}$ with $\vect{y}\in \K\setminus\{\vect{0}\}$. For small $\varepsilon>0$, we maintain
\begin{equation*}
\vect{l}+(1\pm \varepsilon)\vect{y} \in \bd(\vect{l}+\K)\cap \intr(\vect{u}-\K)\setminus \{\vect{l}\},
\end{equation*}
and the optimality of $\vect{x}^*$ implies $\la \vect{c},\vect{y}\ra =0$. But then $\la \vect{c},\vect{x}^*\ra = \la \vect{c}, \vect{l}\ra$ and we can choose $\vect{l}$ as maximizer instead.
\end{proof}

We can also look at the dual problem to get some insight into the problem structure.

\begin{lemma}\label{lem:dual}
Let $\vect{c}=\vect{c}_+ + \vect{c}_-$ be the Moreau decomposition where 
\begin{equation*}
\vect{c}_+ = \pi_{\K}(\vect{c})\quad \text{ and }\quad \vect{c}_- = \pi_{-\K^*}(\vect{c})=-\pi_{\K}(-\vect{c}),
\end{equation*}
since $\K$ is self-dual. Then the dual of \eqref{eq:assumption} is equivalent to
\begin{equation}
\min\left\{ \la \vect{l}-\vect{u}, \vect{y}\ra  \I \vect{y}\eqk -\vect{c}_+,\;  \vect{y}\eqk \vect{c}_-\right\}.\label{eq:simple_dual}
\end{equation}
\end{lemma}
\begin{proof}
The dual problem reads
\begin{equation}
\min \left\{ \la \vect{l},\vect{y}_2\ra - \la \vect{u}, \vect{y}_1\ra   \I \vect{y}_2-\vect{y}_1=\vect{c},\; \vect{y}_1,\vect{y}_2\in \K\right\},\label{eq:dual}
\end{equation}
since $\K$ is self-dual. We can reparametrize $\vect{y}_1=\vect{y}-\vect{c}_-$ and $\vect{y}_2=\vect{y}+\vect{c}_+$ to satisfy the equality constraint and get
\begin{equation*}
\min\left\{ \la \vect{l}-\vect{u}, \vect{y}\ra +\la \vect{l},\vect{c}_+\ra +\la \vect{u},\vect{c}_-\ra  \I \vect{y}\eqk -\vect{c}_+,\, \vect{y}\eqk \vect{c}_-\right\}.
\end{equation*}
Since $\la \vect{l},\vect{c}_+\ra +\la \vect{u},\vect{c}_-\ra$ is constant, the result follows.
\end{proof}

The importance of Lemma~\ref{lem:dual} comes from the following observation.

\begin{remark}\label{cor:join}
The optimal solution of the dual \eqref{eq:simple_dual} is necessarily a least upperbound on the set $\{\vect{c}_-, -\vect{c}_+\}$ in the partial ordered set $(\R^n,\keq)$. Thus, if $(\R^n,\keq)$ is a lattice in the sense of order theory, then the solution of the dual can be recovered from the join $(\vect{c}_-\vee -\vect{c}_+)$ in $(\R^n,\keq)$.
\end{remark}

\subsection{The nonnegative orthant \texorpdfstring{$\K=\R^n_+$}{LP}}

In the case of nonnegative vectors, all our results fall back to the original paper~\cite{Fujishige08}. Compared to the general setting outlined in the previous chapters, it can be shown that the minimum-norm-point algorithm converges to the optimal solution in a finite number of iterations \cite{Wolfe76} and that the LP-Newton method converges in a finite number of steps as well \cite{Fujishige08}, making the proposed framework an overall finite algorithm for the case of $\K=\R^n_+$.

\noindent Of course, $\R^n_+$ is suitable for the CLP-Newton method, as a solution $\vect{x}^*$ is given by greedily choosing the largest increase of the objective function by setting
\begin{equation}
x^*_i = \begin{cases} l_i & \text{if }c_i <0,\\ u_i & \text{else.} \end{cases}\label{eq:greedyLP}
\end{equation}

\noindent In particular, since $\leq $ has the lattice property, Remark~\ref{cor:join} shows $\vect{y}=\vect{0}$ in \eqref{eq:simple_dual}, as
\begin{equation*}
y_i = (\vect{c}_- \vee -\vect{c}_+)_i = \max \{ 0, -|c_i|\} = 0.
\end{equation*}
This confirms \eqref{eq:greedyLP} through the dual variables $\vect{y}_1 = -\pi_{\R^n_-}(\vect{c})$ and $\vect{y}_2 = \pi_{\R^n_+}(\vect{c})$ in \eqref{eq:dual}.

\subsection{The Lorentz-cone \texorpdfstring{$\K=\mc{L}_n$}{SOCP}}

We will denote by 
\begin{equation*}
\mc{L}_n= \left\{(x_0,\tilde{\vect{x}})\in \R_+\times \R^n \I \|\tilde{\vect{x}}\|_2\leq x_0\right\}
\end{equation*}
the Lorentz-Cone. This cone is nice in the sense that we explicitly have
\begin{equation*}
\bd(\mc{L}_n) =\left\{ (x_0,\tilde{\vect{x}})\in \R_+\times \R^n \I \|\tilde{\vect{x}}\|_2= x_0\right\},
\end{equation*}
so we can apply Lemma~\ref{lem:boundary}. 

For the rest of this section, for any $\vect{a}=(a_0,a_1,\ldots,a_n)\in \R^{n+1}$, we will use the notation $\vect{a}^\T=(a_0,\tilde{\vect{a}}^\T)$ with $\tilde{\vect{a}}=(a_1,\ldots,a_n)\in \R^n$ and $a_0\in \R$. Furthermore, for any $\vect{w}\in \mc{L}_n$, we will also define the set 
\begin{equation}
\mc{E}(\vect{w})= \left\{ \vect{x}\in \R^{n+1} \I \|\tilde{\vect{x}}\|_2^2=x_0^2,\quad \|\tilde{\vect{w}}-\tilde{\vect{x}}\|_2^2=(w_0-x_0)^2\right\}. \label{eq:ellipsoid_alt}
\end{equation}

\begin{lemma}\label{lem:ellipsoid_representation}
For any $\vect{w}=(w_0, \tilde{\vect{w}})\in \intr(\mc{L}_n)$, define the parameters
\begin{equation*}
\bar{\vect{w}}:=\frac{1}{w_0}\tilde{\vect{w}},\quad 
\bar{w}_0 := \frac{w_0^2-\|\tilde{\vect{w}}\|^2_2}{2w_0}, 
\quad \matr{Q}:=\matr{I}_n-\bar{\vect{w}}\bar{\vect{w}}^\T\quad
\text{and}\quad  \gamma:=\sqrt{\tfrac{1}{2}w_0 \bar{w}_0}.
\end{equation*}
Then 
\begin{equation*}
\mc{E}(\vect{w})= \left\{ \vect{x}\in \R^{n+1} \,\middle|\, x_0 = \la \tilde{\vect{x}}, \bar{\vect{w}}\ra + \bar{w}_0  ,\; \left\| \matr{Q}^{\frac{1}{2}} (\tilde{\vect{x}}-\tfrac{1}{2}\tilde{\vect{w}})\right\|_2^2=\gamma^2 \right\}
\end{equation*}
and in particular, $\mc{E}(\vect{w})$ is an n-dimensional ellipsoid.
\end{lemma}
\begin{proof}
Subtracting the equations in \eqref{eq:ellipsoid_alt} immediately shows containment in the hyperplane
\begin{equation}
H=\left\{\vect{x}\in \R^{n+1} \,\middle|\, x_0 = \la \tilde{\vect{x}}, \bar{\vect{w}}\ra + \bar{w}_0 \right\},\label{eq:hyperplane}
\end{equation}
where $\bar{\vect{w}}$ and $\bar{w}_0$ are well defined since $\vect{w}\in \intr(\mc{L}_n)$. 

Using \eqref{eq:hyperplane} in either equation in \eqref{eq:ellipsoid_alt} on $x_0$ yields an equation of the form
\begin{equation*}
0=\tilde{\vect{x}}^\T \matr{Q} \tilde{\vect{x}} - 2\la\bar{w}_0 \bar{\vect{w}},\tilde{\vect{x}}\ra -\bar{w}_0^2 
\end{equation*}
where $\matr{Q}$ is positive definite since $\|\bar{\vect{w}}\|_2 < 1$. Completing the square yields the equivalent condition
\begin{equation*}
\| \matr{Q}^{\frac{1}{2}}(\tilde{\vect{x}} - \bar{w}_0 \matr{Q}^{-1} \bar{\vect{w}} )\|_2^2 = 
\bar{w}_0^2 +  \bar{w}_0^2\cdot \bar{\vect{w}}^\T \matr{Q}^{-1} \bar{\vect{w}},
\end{equation*}
which defines an $n$-dimensional ellipsoid.

Using the Sherman-Morrison formula we can simplify 
\begin{equation*}
\bar{w}_0 \matr{Q}^{-1}\bar{\vect{w}} = \tfrac{1}{2}\tilde{\vect{w}}
\end{equation*}
and
\begin{equation*}
\bar{w}_0^2+ \bar{w}_0^2\cdot \bar{\vect{w}}^\T \matr{Q}^{-1} \bar{\vect{w}} 
 =\tfrac{1}{2}w_0 \bar{w}_0=\gamma^2.
\end{equation*}
\end{proof}

\begin{lemma}\label{lem:ellipsoid_optimization}
For any $w\in \mc{L}_n$, the problem
\begin{equation}
\max\left\{ \la \vect{c},\vect{x}\ra \,\middle|\, \vect{x}\in \mc{E}(\vect{w})\right\}\label{eq:ellipsoid_optimization}
\end{equation}
can be solved in $\mc{O}(n)$.
\end{lemma}

\begin{proof}
We will distinguish the cases $\vect{w}\in \bd(\mc{L}_n)$ and $\vect{w}\in \intr(\mc{L}_n)$, which can be checked in $\mc{O}(n)$.

For $\vect{w}\in \bd(\mc{L}_n)$, we claim that $\mc{E}(\vect{w})=\conv(\{\vect{0},\vect{w}\})$. This is easy to see since by assumption, $\|\tilde{\vect{w}}\|_2 = w_0$, $\|\tilde{\vect{x}}\|_2=x_0$ and therefore 
\begin{equation*}
\|\tilde{\vect{w}}-\tilde{\vect{x}}\|_2=\|\tilde{\vect{w}}\|_2-\|\tilde{\vect{x}}\|_2 \quad \forall \vect{x}\in \mc{E}(\vect{w}),
\end{equation*}
which is only possible if $\tilde{\vect{x}}$ is a multiple of $\tilde{\vect{w}}$. This also fixes $x_0$ to be the same multiple of $w_0$ and consequently, \eqref{eq:ellipsoid_optimization} will either be $0$ or $\la \vect{c},\vect{w}\ra$. 

For $\vect{w}\in \intr(\mc{L}_n)$, Lemma~\ref{lem:ellipsoid_representation} allows us to use the parametrization
\begin{equation}
\tilde{\vect{x}} = \matr{Q}^{-\frac{1}{2}}\vect{y} +\tfrac{1}{2}\tilde{\vect{w}},\quad s.t.\quad \|\vect{y}\|_2^2=\gamma^2 \label{eq:parametrization}
\end{equation}
of \eqref{eq:ellipsoid_optimization} in terms of $\vect{y}$ with corresponding objective
\begin{align*}
\la \vect{c},\vect{x}\ra
&= \la \tilde{\vect{c}}+\tfrac{c_0}{w_0} \tilde{\vect{w}},\tilde{\vect{x}}\ra + c_0\bar{w}_0 \\
&\equiv \la \tilde{\vect{c}}+\tfrac{c_0}{w_0} \tilde{\vect{w}}, \matr{Q}^{-\frac{1}{2}}\vect{y}+\tfrac{1}{2}\tilde{\vect{w}}\ra
\equiv \la \matr{Q}^{-\frac{1}{2}}(\tilde{\vect{c}}+\tfrac{c_0}{w_0} \tilde{\vect{w}}), \vect{y}\ra,
\end{align*}
where $\equiv$ denotes equality up to a constant difference. We need to distinguish two cases:

If $\matr{Q}^{-\frac{1}{2}}(\tilde{\vect{c}}+\frac{c_0}{w_0}\tilde{\vect{w}})=\vect{0}$, then the optimal solution $\vect{y}^*$ can be chosen arbitrarily and we set $\vect{y}^*=\frac{\gamma}{n}\vect{e}$, where $\vect{e}$ is the vector of all ones. 

Otherwise, since $\vect{y}$ is chosen from a scaled Euclidean ball, the optimal parametrization $\vect{y}^*$ is parallel to the new objective and we get the closed form expression
\begin{equation*}
\vect{y}^*= \gamma \cdot \frac{\matr{Q}^{-\frac{1}{2}}(\tilde{\vect{c}}+\tfrac{c_0}{w_0} \tilde{\vect{w}})}{\|\matr{Q}^{-\frac{1}{2}}(\tilde{\vect{c}}+\tfrac{c_0}{w_0} \tilde{\vect{w}})\|_2}.
\end{equation*}
In either case, we can use the parametrization \eqref{eq:parametrization} to recover the optimal solution 
\begin{equation}
\tilde{\vect{x}}^* = \matr{Q}^{-\frac{1}{2}} \vect{y}^* + \tfrac{1}{2}\tilde{\vect{w}},\quad 
x_0^* = \la \tilde{\vect{x}}^*, \bar{\vect{w}}\ra + \bar{w}_0.\label{eq:optimal_y}
\end{equation}

In order to show that this expression can be evaluated in linear time, we only need to show that multiplication by $\matr{Q}^{-\frac{1}{2}}$ can be done in $\mc{O}(n)$. Therefore, we will proceed by giving an explicit formula for $\matr{Q}^{-\frac{1}{2}}$. 

If $\tilde{\vect{w}}\neq \vect{0}$, we set
\begin{equation*}
\alpha :=  \frac{1-\tfrac{2\gamma}{w_0}}{\|\tilde{\vect{w}}\|_2^2},\quad \beta := \frac{\alpha}{1-\alpha \|\tilde{\vect{w}}\|_2^2} = \frac{\tfrac{w_0}{2\gamma}-1}{\|\tilde{\vect{w}}\|_2^2},
\end{equation*}
and $\alpha=\beta=0$ otherwise. Using $w_0>\|\tilde{\vect{w}}\|_2$, one can verify that $\alpha,\beta\geq 0$ and a
straightforward computation shows the identities
\begin{equation}
\matr{Q}^{\frac{1}{2}}=\matr{I}_n-\alpha\tilde{\vect{w}}\tilde{\vect{w}}^\T,\quad \matr{Q}^{-\frac{1}{2}}=\matr{I}_n+\beta\tilde{\vect{w}}\tilde{\vect{w}}^\T,\label{eq:explicit_Q}
\end{equation}
where one identity can be reduced to the other by the Sherman-Morrison formula.
Finally, \eqref{eq:explicit_Q}~shows
\begin{equation*}
\matr{Q}^{-\frac{1}{2}}\tilde{\vect{x}} = \tilde{\vect{x}} + \beta \la \tilde{\vect{x}}, \tilde{\vect{w}}\ra \tilde{\vect{w}}\quad \forall \tilde{\vect{x}}\in\R^n
\end{equation*}
where the right side can be computed in $\mc{O}(n)$.
\end{proof}

\begin{theorem}\label{thm:maxLn}
The cone $\mc{L}_n$ is suitable for the CLP-Newton method.
\end{theorem}
\begin{proof}
Through translation we can assume that $\vect{l}=\vect{0}$ and focus on the case 
\begin{equation}
\max\left\{ \la \vect{c},\vect{x}\ra \I \vect{x}\in [\vect{0}, \vect{w}]_{\mc{L}_n}\subseteq \R^{n+1}\right\}\label{eq:maxLn}
\end{equation}
where $\vect{w}=\vect{u}-\vect{l}\in \mc{L}_n\setminus \{\vect{0}\}$ and consequently $w_0 = u_0-l_0>0$. By using Lemma~\ref{lem:boundary}, it suffices to compute 
\begin{equation}
\max\left\{ \la \vect{c},\vect{x}\ra \,\middle|\, \vect{x}\in \bd(\mc{L}_n)\cap \bd(\vect{w}-\mc{L}_n)=:\mc{E}'(\vect{w})\right\} \label{eq:lorentz_optimization}
\end{equation}
and compare this value to $\la \vect{c},\vect{0}\ra=0$ and $\la \vect{c},\vect{w}\ra$.  We thus have
\begin{align*}
\mc{E}'(\vect{w})&= \left\{ \vect{x}\in \R^{n+1} \,\middle|\, \|\tilde{\vect{x}}\|_2=x_0, \|\tilde{\vect{w}}-\tilde{\vect{x}}\|_2=w_0-x_0\right\}\\
&=\left\{\vect{x}\in \mc{E}(\vect{w}) \,\middle|\, x_0\in [0,w_0]\right\}
\end{align*}
by \eqref{eq:ellipsoid_alt} and claim that $\mc{E}'(\vect{w})=\mc{E}(\vect{w})$. 

To see this, we can use Lemma~\ref{lem:ellipsoid_optimization} with objective $\vect{c}^\T=(\pm 1, \vect{0}^\T)$ to get 
\begin{equation*}
\max \left\{ \pm x_0 \I \vect{x}\in \mc{E}(\vect{w})\right\} = \tfrac{1}{2}w_0 \pm \tfrac{1}{2} \|\tilde{\vect{w}}\|_2\in [0,w_0],
\end{equation*}
where the bounds follow from $\vect{w}\in \mc{L}_n$.
All that is left now is to compute 
\begin{equation*}
\max\left\{ \la \vect{c},\vect{x}\ra \,\middle|\, \vect{x}\in \mc{E}(\vect{w})\right\},
\end{equation*}
which can be done in linear time according to Lemma~\ref{lem:ellipsoid_optimization}.
\end{proof}

\begin{remark}
The parameters $\bar{w}_0$, $\beta$ and $\gamma$ only depend on $\vect{w}=\vect{u}-\vect{l}$. When optimizing multiple times over $[\vect{l},\vect{u}]_{\K}$ with different objective functions, like in our setting, these parameters can be stored and need only be computed once.
\end{remark}

We close this section with Algorithm~\ref{alg:maxLn}, an explicit linear time algorithm for solving \eqref{eq:maxLn} according to the preceding theorem.

\begin{algorithm}[H]\label{alg:maxLn}
\caption{Explicit solution to Problem \eqref{eq:maxLn} } 
\SetKw{NOT}{not}
\SetKw{OR}{or}

\SetKwFunction{F}{F}

\KwData{Data $\vect{l}, \vect{u}, \vect{c}$ for Problem \eqref{eq:maxLn}, $\bar{w}_0, \beta, \gamma$ as in Theorem~\ref{thm:maxLn}.}
\KwResult{Solution $\vect{x}^*\in [\vect{l},\vect{u}]_{\K}$ to \eqref{eq:maxLn}.}
$\vect{w}=\vect{u}-\vect{l}$\;
$\vect{x}^*=\vect{0}$\;
\If{$\|\tilde{\vect{w}}\|_2^2<w_0^2$}{
$\tilde{\vect{y}}^* = \tilde{\vect{c}}+\frac{c_0}{w_0}\tilde{\vect{w}}$\;
$\tilde{\vect{y}}^* = \tilde{\vect{y}}^* + \beta \la \tilde{\vect{y}}^*,\tilde{\vect{w}}\ra \tilde{\vect{w}} $\;
\If{$\tilde{\vect{y}}^* = \vect{0}$}{
$\tilde{\vect{y}}^* = \frac{\gamma}{n}\vect{e}$\;
}
\Else{
$\tilde{\vect{y}}^* = \frac{\gamma}{\|\tilde{\vect{y}}^*\|_2}  \tilde{\vect{y}}^*$\;
}
$\tilde{\vect{x}}^* = \tilde{\vect{y}}^* + \beta \la \tilde{\vect{y}}^*,\tilde{\vect{w}}\ra \tilde{\vect{w}} $\;
$\tilde{\vect{x}}^* = \tilde{\vect{x}}^* + \frac{1}{2} \tilde{\vect{w}}$\;
$\vect{x}^*_0 = \frac{1}{w_0}\la \tilde{\vect{x}}^*, \tilde{\vect{w}}\ra + \bar{w}_0$\;
}
$\vect{x}^*  = \vect{x}^*+\vect{l}$\;
\Return $\argmax\left\{ \la \vect{y},\vect{c}\ra \I  \vect{y}\in \{\vect{l},\vect{x}^*,\vect{u}\} \right\}$\;
\end{algorithm}

\subsection{The positive semidefinite cone \texorpdfstring{$\K=\mc{S}^n_+$}{SDP}}

Let $\mc{S}^n_+$ denote the cone of symmetric $n\times n$ matrices that are positive semidefinite and let 
$\preceq$ be the corresponding conic order. It is important to note for statements about complexity that we can embed $\mc{S}^n_+\subseteq \R^N$ for $N=\binom{n}{2}$. 

Now \eqref{eq:assumption} reads
\begin{equation*}
\max\left\{ \la \matr{C},\matr{X}\ra \I \matr{L}\preceq \matr{X} \preceq \matr{U}\right\}.
\end{equation*}
Instead of treating this problem directly, we will perform a preprocessing step. We first use the substitution $\matr{Y}=\matr{X}-\matr{L}$ to get the equivalent problem
\begin{equation*}
\max\left\{ \la \matr{C},\matr{Y}\ra \I \matr{0}\preceq \matr{Y} \preceq \matr{U}-\matr{L}=:\matr{W}\right\}
\end{equation*}
where we dropped the constant $\la \matr{C},\matr{L}\ra$ from the objective. In the following, we will assume that $\matr{W}\in \intr(\mc{S}^n_+)$ to simplify the argument (the following can be adapted for the case where $\matr{W}$ is singular). 
Using the Cholesky decomposition $\matr{W} = \matr{VV}^\T$, we can rewrite $\matr{Y}=\matr{VZV}^\T$ to get the equivalent problem
\begin{equation}
\max\left\{ \la \matr{C}',\matr{Z}\ra \I \matr{0}\preceq \matr{Z} \preceq \matr{I}_n\right\},\label{eq:greedySDP}
\end{equation}
where $\matr{C}'=\matr{V^\T CV}$.
In particular, the conic constraints reduce to box-constraints on the eigenvalues of $Z$ and we can now solve the problem explicitly.

\begin{theorem}
Let $\matr{C}'=\matr{B^\T D B}$ be the eigenvalue decomposition of $\matr{C}'$. Then the solution to~\eqref{eq:greedySDP} is given by
\begin{equation*}
\matr{Z} = \matr{B^\T \Lambda^* B}
\end{equation*}
where $\matr{\Lambda}^*$ is a diagonal matrix with $\diag(\matr{\Lambda}^*)=\vect{\lambda}^*$ and $\vect{\lambda}^*$ is the solution of 
\begin{equation}
\max\left\{ \la \diag(\matr{D}), \vect{\lambda} \ra \I \vect{\lambda} \in  [0,1]^n \right\}.\label{eq:greedySDPLP}
\end{equation}

\end{theorem}
\begin{proof}
Let $\vect{d}=\diag(\matr{D})$ and let $\vect{\lambda}$ denote the eigenvalues of $\matr{Z}$. Then the Hoffman-Wielandt inequality states
\begin{equation*}
\la \matr{C}',\matr{Z}\ra \leq \la \vect{d}, \matr{P} \vect{\lambda}\ra ,
\end{equation*}
where $\matr{P}$ is a permutation that assigns the $i$-th biggest entry of $\vect{\lambda}$ to the $i$-th biggest entry of $\vect{d}$ for all $i\in[n]$. Since $\vect{\lambda} \in [0,1]^n$, the right hand side is maximal when $\vect{\lambda}$ is the solution $\vect{\lambda}^*$ of \eqref{eq:greedySDPLP} and $\matr{P}$ the identity. Then one can verify that the left hand side also attains this upperbound by choosing $\matr{Z}=\matr{B}^\T \Diag(\vect{\lambda}^*)\matr{B}$.
\end{proof}

By solving \eqref{eq:greedyLP}, we are able to solve \eqref{eq:greedySDP} as well, but the actual computational burden lies in the corresponding reduction. To this end, we have the following result.

\begin{theorem}
The cone $\mc{S}^n_+$ is suitable for the \normalfont{CLP}-Newton method. 
\end{theorem}
\begin{proof}
The complexity of computing the eigenvalues of a $n\times n$ matrix as well as matrix computation is contained in $\mathcal{O}(n^3)$. Since $S^n_+\subseteq \R^N$ with $n\in \mathcal{O}(N^{1/2})$, we get an algorithm in $\mathcal{O}(n^3)\subseteq \mathcal{O}(N^{3/2})\subseteq \mathcal{O}(N^2)$.
\end{proof}

\begin{remark}
The preceding theorem may seem remarkable in terms of Remark~\ref{cor:join}, since the conic order $\preceq$ induced by $\mc{S}^n_+$ is explicitly known \emph{not} to define a lattice. 
In particular, if we have a proper interval $\matr{L}\prec \matr{U}$, then Slater's condition holds and we expect strong duality to hold in Lemma~\ref{lem:dual}, so that the preceding theorem yields an oracle for elements of the set of least upperbounds of $\{\matr{C}_-,-\matr{C}_+\}$ in $(\R^N,\preceq)$.
\end{remark}

\section{Experiments}

In this section we show some experiments done with a simple implementation of the CLP-Newton method. For this, we used MATLAB version 8.1.0.604 (R2013) with an Intel i5 of 3.2 GHz $\times$ 4 and 16 GB of memory. 

As a reference, we used the widespread SDPT3 package \cite{TTT03}.

\subsection{Data generation}

\textbf{SOCP}

\noindent Based on parameter tuples $\frac{n}{m}$, we generated random instances for $\K=\mc{L}_n$. We set $\vect{l}=\vect{0}$, $u_0=10$ and $\tilde{\vect{u}}$ to a random vector with entries in~$[-0.5, 0.5]$, which was afterwards normalized such that $\|\tilde{\vect{u}}\|_2$ was a random number in the interval $[0,10]$. 

The vector $\vect{c}$ was randomly chosen with entries in~$[-0.5, 0.5]$ and $\matr{A}$ was chosen as a random $m\times n$ matrix with entries in~$[0,1]$. To guarantee feasibility, we included the midpoint of $[\vect{l},\vect{u}]_{\mc{L}_n}$ into the feasible region by setting $\vect{b}=\frac{1}{2}\matr{A}(\vect{u}-\vect{l})$.
\medskip

\noindent\textbf{SDP}

\noindent Based on parameter tuples $\frac{n}{m}$, we generated random instances for $\K=\mc{S}^n_+$. We set $\matr{L}=\matr{0}$ and construct $\matr{U}$ in the following way:
We first construct a random $n\times n$ matrix $\matr{V}$ with values in $[0,1]$ and set $\matr{U}=\matr{VV}^{\T}+\frac{1}{10}\matr{I}_n$. Afterwards, $\matr{U}$ is normalized such that $\tr(\matr{U})=10$.

The remaining parameters are chosen in the same way as for SOCP: $\matr{C}$ was randomly chosen with entries in~$[-0.5, 0.5]$ and $\mc{A}$ was chosen as a random $m\times n^2$ linear operator with entries in~$[0,1]$. To guarantee feasibility, we included the midpoint of $[\matr{L},\matr{U}]_{\mc{S}^n_+}$ into the feasible region by setting $\matr{B}=\frac{1}{2}\mc{A}(\matr{U}-\matr{L})$.

\subsection{Plots}

In the following plots, each data point $\frac{n}{m}$ corresponds to the average of 25 instances randomly generated according to the procedure outlined before with parameters $\frac{n}{m}$. The error tolerance for CLPN was set to $10^{-6}$ and the error tolerance $\varepsilon$ given in the plots apply to the subroutine MNP.

\noindent
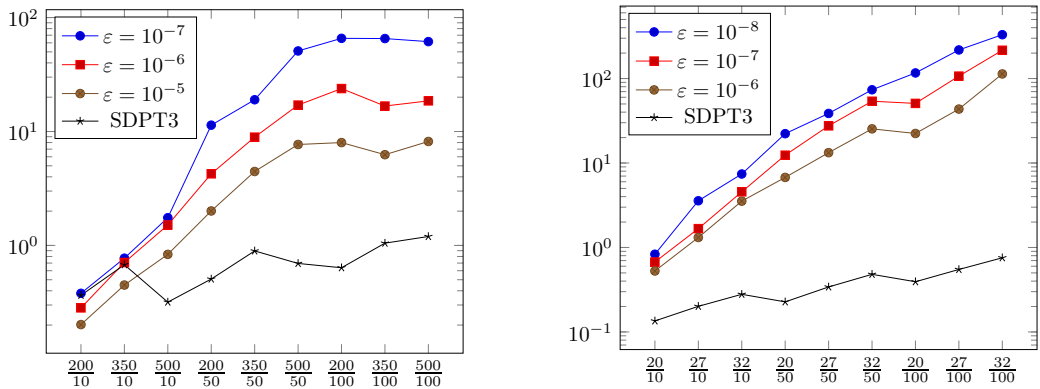
\begin{figure}[!ht]

\begin{minipage}{.5\textwidth}
\centering
\begin{tikzpicture}[scale=0.8]
\begin{axis}[
	legend entries={$\varepsilon=10^{-7}$,$\varepsilon=10^{-6}$,$\varepsilon=10^{-5}$, SDPT3},
	legend style={at={(0.02,0.98)},anchor=north west},
	ymode = log,
    xtick={1,2,3,4,5,6,7,8,9},
    xticklabels={$\frac{200}{10}$,$\frac{350}{10}$,$\frac{500}{10}$,
			     $\frac{200}{50}$,$\frac{350}{50}$,$\frac{500}{50}$,
    			 $\frac{200}{100}$,$\frac{350}{100}$,$\frac{500}{100}$}
]
\addplot table [x=Inst, y=eps7] {Data/SOCP_Time_Table};
\addplot table [x=Inst, y=eps6] {Data/SOCP_Time_Table};
\addplot table [x=Inst, y=eps5] {Data/SOCP_Time_Table};
\addplot table [x=Inst, y=SDPT3]{Data/SOCP_Time_Table};
\end{axis}
\end{tikzpicture}
\end{minipage}%
\begin{minipage}{.5\textwidth}
\centering
\begin{tikzpicture}[scale=0.8]
\begin{axis}[
	legend entries={$\varepsilon=10^{-8}$, $\varepsilon=10^{-7}$,$\varepsilon=10^{-6}$, SDPT3},
	legend style={at={(0.02,0.98)},anchor=north west},
	ymode = log,
    xtick={1,2,3,4,5,6,7,8,9},
    xticklabels={$\frac{20}{10}$,$\frac{27}{10}$,$\frac{32}{10}$,
			     $\frac{20}{50}$,$\frac{27}{50}$,$\frac{32}{50}$,
    			 $\frac{20}{100}$,$\frac{27}{100}$,$\frac{32}{100}$}
]
\addplot table [x=Inst, y=eps8] {Data/SDP_Time_Table};
\addplot table [x=Inst, y=eps7] {Data/SDP_Time_Table};
\addplot table [x=Inst, y=eps6] {Data/SDP_Time_Table};
\addplot table [x=Inst, y=SDPT3]{Data/SDP_Time_Table};
\end{axis}
\end{tikzpicture}
\end{minipage}

\caption{Runtime (sec) for parameters $\frac{n}{m}$. Left: SOCP. Right: SDP.}
\label{fig:runtime}
\end{figure}

\begin{figure}[!ht]

\noindent
\begin{minipage}{.5\textwidth}
\centering
\begin{tikzpicture}[scale=0.8]
\begin{axis}[
	legend entries={$\varepsilon=10^{-7}$,$\varepsilon=10^{-6}$,$\varepsilon=10^{-5}$},
	legend style={at={(0.02,0.98)},anchor=north west},
    xtick={1,2,3,4,5,6,7,8,9},
    xticklabels={$\frac{200}{10}$,$\frac{350}{10}$,$\frac{500}{10}$,
			     $\frac{200}{50}$,$\frac{350}{50}$,$\frac{500}{50}$,
    			 $\frac{200}{100}$,$\frac{350}{100}$,$\frac{500}{100}$}
]
\addplot table [x=Inst, y=eps7] {Data/SOCP_Nstep_Table};
\addplot table [x=Inst, y=eps6] {Data/SOCP_Nstep_Table};
\addplot table [x=Inst, y=eps5] {Data/SOCP_Nstep_Table};
\end{axis}
\end{tikzpicture}
\end{minipage}%
\begin{minipage}{.5\textwidth}
\centering
\begin{tikzpicture}[scale=0.8]
\begin{axis}[
	legend entries={$\varepsilon=10^{-8}$,$\varepsilon=10^{-7}$,$\varepsilon=10^{-6}$},
	legend style={at={(0.02,0.98)},anchor=north west},
    xtick={1,2,3,4,5,6,7,8,9},
    xticklabels={$\frac{20}{10}$,$\frac{27}{10}$,$\frac{32}{10}$,
			     $\frac{20}{50}$,$\frac{27}{50}$,$\frac{32}{50}$,
    			 $\frac{20}{100}$,$\frac{27}{100}$,$\frac{32}{100}$}	
]
\addplot table [x=Inst, y=eps8] {Data/SDP_Nstep_Table};
\addplot table [x=Inst, y=eps7] {Data/SDP_Nstep_Table};
\addplot table [x=Inst, y=eps6] {Data/SDP_Nstep_Table};
\end{axis}
\end{tikzpicture}
\end{minipage}

\caption{Newton-steps for parameters $\frac{n}{m}$. Left: SOCP. Right: SDP.}
\label{fig:Newton}
\end{figure}
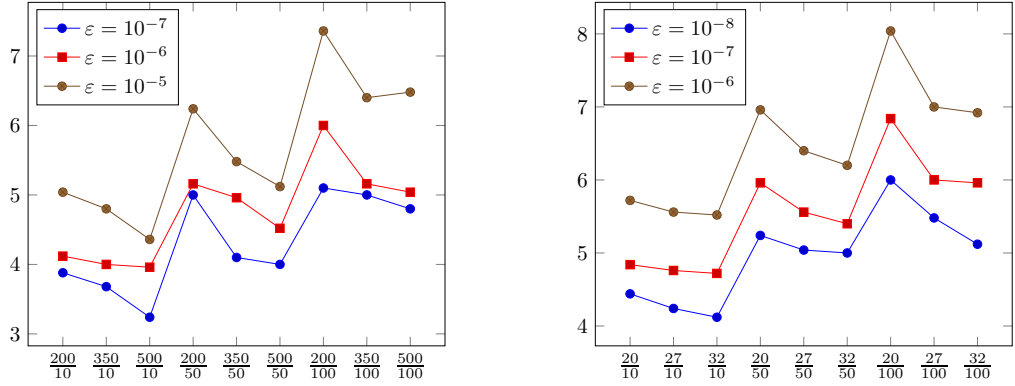

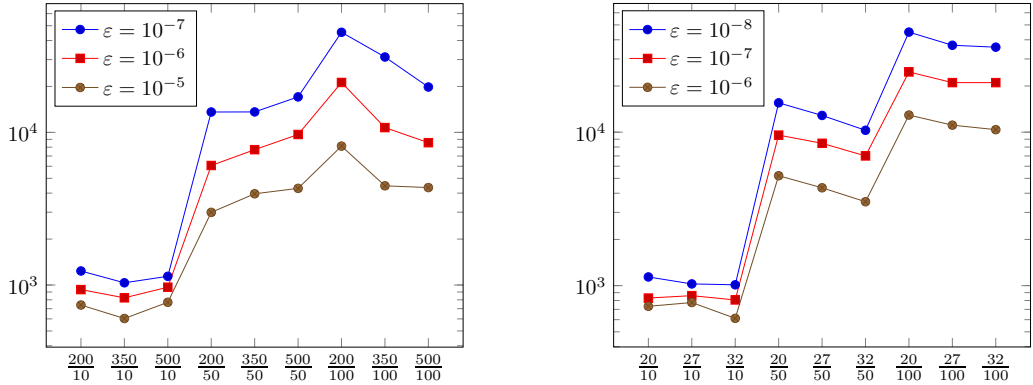
\begin{figure}[!ht]

\noindent
\begin{minipage}{.5\textwidth}
\centering
\begin{tikzpicture}[scale=0.8]
\begin{axis}[
	legend entries={$\varepsilon=10^{-7}$,$\varepsilon=10^{-6}$,$\varepsilon=10^{-5}$},
	legend style={at={(0.02,0.98)},anchor=north west},
	ymode = log,
    xtick={1,2,3,4,5,6,7,8,9},
    xticklabels={$\frac{200}{10}$,$\frac{350}{10}$,$\frac{500}{10}$,
			     $\frac{200}{50}$,$\frac{350}{50}$,$\frac{500}{50}$,
    			 $\frac{200}{100}$,$\frac{350}{100}$,$\frac{500}{100}$}
]
\addplot table [x=Inst, y=eps7] {Data/SOCP_FWstep_Table};
\addplot table [x=Inst, y=eps6] {Data/SOCP_FWstep_Table};
\addplot table [x=Inst, y=eps5] {Data/SOCP_FWstep_Table};
\end{axis}
\end{tikzpicture}
\end{minipage}%
\begin{minipage}{.5\textwidth}
\centering
\begin{tikzpicture}[scale=0.8]
\begin{axis}[
	legend entries={$\varepsilon=10^{-8}$,$\varepsilon=10^{-7}$,$\varepsilon=10^{-6}$},
	legend style={at={(0.02,0.98)},anchor=north west},
	ymode = log,
    xtick={1,2,3,4,5,6,7,8,9},
    xticklabels={$\frac{20}{10}$,$\frac{27}{10}$,$\frac{32}{10}$,
			     $\frac{20}{50}$,$\frac{27}{50}$,$\frac{32}{50}$,
    			 $\frac{20}{100}$,$\frac{27}{100}$,$\frac{32}{100}$}	
]
\addplot table [x=Inst, y=eps8] {Data/SDP_FWstep_Table};
\addplot table [x=Inst, y=eps7] {Data/SDP_FWstep_Table};
\addplot table [x=Inst, y=eps6] {Data/SDP_FWstep_Table};
\end{axis}
\end{tikzpicture}

\end{minipage}

\caption{MNP computations for parameters $\frac{n}{m}$. Left: SOCP. Right: SDP.}
\label{fig:MNP}
\end{figure}

The plots in Figure~\ref{fig:runtime} show that the choice of accuracy for MNP has a great impact on the overal running time of the algorithm. While reducing the accuracy can speed up the algorithm significantly, going below the accuracy given in the plots often resulted in major problems in converging to the solution, so care has to be taken by choosing this parameter. 

Overall, the data in Figure~\ref{fig:Newton} resembles the results of \cite{Fujishige08} for the case of $\R^n_+$, in the sense that only a few Newton-steps are necessary to get a close approximate solution. Figure~\ref{fig:MNP} also shows that, like in the original paper, the number $m$ of constraints seems to a have a much larger impact on the performance than the number of the variables $n$, since much more subproblems have to be solved. 

\section{Conclusion}

In this paper, we have shown that the CLP-Newton method can be successfully used to solve CLPs over $\K$-zonotopes for a proper, convex self-dual cone $\K$. In particular, this resulted in a new method for CLPs, which is parallelizable if $\K$ can be decomposed into a Cartesian product of smaller cones.
We also introduced the concept of $\K$-zonotopes and commented on some of their properties in terms of optimization, which might be interesting objects in their own right.

As an application of the framework, we gave explicit algorithms to solve linear problems over $\mc{L}_n$-zonotopes and $\mc{S}_+^n$-zonotopes and examined how they perform in experiments. 

Since approximating the minimum-norm-point problem for general cones with a Frank-Wolfe algorithm is apparently much less efficient than for the non-negative cone, our implementation was slow compared to the reference algorithm. In particular, since the number of Newton-steps remains small even for more complex cones than $\R^n_+$, any improvement for the minimum-norm-point subroutine would result in a much better overall algorithm. To this end, it would be interesting to have more control about the lowest necessary precision for MNP, since a higher precision tends to bloat the runtime unnecessarily, as shown in the experiments.

An interesting question for further research will be whether the CLP-Newton method can be improved to compete with interior method on special structures and how parallelization can be applied successfully.

\bibliographystyle{plain}
\bibliography{papers}

\end{document}